\documentclass[twoside,american]{amsart}
\usepackage[T1]{fontenc}
\usepackage[latin9]{inputenc}
\usepackage{color}
\usepackage[english]{babel}
\usepackage{refstyle}
\usepackage{units}
\usepackage{amsthm}
\usepackage{amsmath}
\usepackage{amssymb}
\usepackage[unicode=true,pdfusetitle,
 bookmarks=true,bookmarksnumbered=false,bookmarksopen=false,
 breaklinks=false,pdfborder={0 0 0},backref=false,colorlinks=true]
 {hyperref}
\hypersetup{
 pdfstartview=FitB,pdfstartpage=1,bookmarksopen}

\makeatletter

\AtBeginDocument{\providecommand\thmref[1]{\ref{thm:#1}}}
\AtBeginDocument{\providecommand\secref[1]{\ref{sec:#1}}}
\AtBeginDocument{\providecommand\propref[1]{\ref{prop:#1}}}
\AtBeginDocument{\providecommand\eqref[1]{\ref{eq:#1}}}
\AtBeginDocument{\providecommand\lemref[1]{\ref{lem:#1}}}
\RS@ifundefined{subref}
  {\def\RSsubtxt{section~}\newref{sub}{name = \RSsubtxt}}
  {}
\RS@ifundefined{thmref}
  {\def\RSthmtxt{theorem~}\newref{thm}{name = \RSthmtxt}}
  {}
\RS@ifundefined{lemref}
  {\def\RSlemtxt{lemma~}\newref{lem}{name = \RSlemtxt}}
  {}

\theoremstyle{plain}
\numberwithin{equation}{section}
\numberwithin{figure}{section}
\numberwithin{table}{section}
    \newtheorem{stthm}{\protect\theoremname}[section]
    
    \newtheorem{spprop}{\protect\propositionname}[section]

 \newref{sec}{%
 name = \RSsectxt,
 names = \RSsecstxt,
 lsttxt = \RSlsttxt,
 lsttwotxt = \RSlsttwotxt,
 refcmd = \ref{#1}}
 \newref{thm}{%
 name = \theoremname~,
 lsttxt = \RSlsttxt,
 lsttwotxt = \RSlsttwotxt,
 refcmd = {\ref{#1}}}
 \usepackage{enumitem}
 \usepackage{enumitem}
 \newref{prop}{%
 name = \propositionname~,
 lsttxt = \RSlsttxt,
 lsttwotxt = \RSlsttwotxt,
 refcmd = {\ref{#1}}}
 \theoremstyle{plain}
 \newtheorem*{cor*}{\protect\corollaryname}
 \newref{cor}{%
 name = \corollaryname~,
 lsttxt = \RSlsttxt,
 lsttwotxt = \RSlsttwotxt,
 refcmd = {\ref{#1}}}
 \theoremstyle{remark}
 \newtheorem{srrem}{\protect\remarkname}[section]
 \newref{rem}{%
 name = \remarkname~,
 lsttxt = \RSlsttxt,
 lsttwotxt = \RSlsttwotxt,
 refcmd = {\ref{#1}}}
 \theoremstyle{plain}
 \newtheorem{sllem}{\protect\lemmaname}[section]
 \newref{lem}{%
 name = \lemmaname~,
 lsttxt = \RSlsttxt,
 lsttwotxt = \RSlsttwotxt,
 refcmd = {\ref{#1}}}
 \theoremstyle{remark}
 \newtheorem*{rem*}{\protect\remarkname}
 \newref{rem}{%
 name = \remarkname~,
 lsttxt = \RSlsttxt,
 lsttwotxt = \RSlsttwotxt,
 refcmd = {\ref{#1}}}
 \theoremstyle{plain}
 \newtheorem*{thm*}{\protect\theoremname}
 \newref{thm}{%
 name = \theoremname~,
 lsttxt = \RSlsttxt,
 lsttwotxt = \RSlsttwotxt,
 refcmd = {\ref{#1}}}

\newcommand\mynobreakpar{\par\nobreak\@afterheading}

\makeatother

 \providecommand{\corollaryname}{Corollary}
 \providecommand{\lemmaname}{Lemma}
 \providecommand{\propositionname}{Proposition}
 \providecommand{\remarkname}{Remark}
 \providecommand{\theoremname}{Theorem}

\begin{document}
\def\rightmark{ON ESTIMATES FOR WEIGHTED BERGMAN PROJECTIONS}
\def\leftmark{P. CHARPENTIER, Y. DUPAIN \& M. MOUNKAILA}
\def\RSsectxt{Section~}%

\title{On estimates for weighted Bergman projections}

\author{P. Charpentier, Y. Dupain \& M. Mounkaila}
\begin{abstract}
In this note we show that the weighted $L^{2}$-Sobolev estimates
obtained by P. Charpentier, Y. Dupain \& M. Mounkaila for the weighted
Bergman projection of the Hilbert space $L^{2}\left(\Omega,d\mu_{0}\right)$
where $\Omega$ is a smoothly bounded pseudoconvex domain of finite
type in $\mathbb{C}^{n}$ and $\mu_{0}=\left(-\rho_{0}\right)^{r}d\lambda$,
$\lambda$ being the Lebesgue measure, $r\in\mathbb{Q}_{+}$ and $\rho_{0}$
a special defining function of $\Omega$, are still valid for the
Bergman projection of $L^{2}\left(\Omega,d\mu\right)$ where $\mu=\left(-\rho\right)^{r}d\lambda$,
$\rho$ being any defining function of $\Omega$. In fact a stronger directional Sobolev
estimate is established. Moreover similar generalizations are
obtained for weighted $L^{p}$-Sobolev and lipschitz estimates in
the case of pseudoconvex domain of finite type in $\mathbb{C}^{2}$
and for some convex domains of finite type.
\end{abstract}



\address{P. Charpentier, Universit\'e Bordeaux I, Institut de Math\'ematiques
de Bordeaux, 351, Cours de la Lib\'eration, 33405, Talence, France}

\address{M. Mounkaila, Universit\'e Abdou Moumouni, Facult\'e des Sciences,
B.P. 10662, Niamey, Niger}

\email{P. Charpentier: philippe.charpentier@math.u-bordeaux1.fr}

\email{M. Mounkaila: modi.mounkaila@yahoo.fr}

\maketitle

\section{Introduction}

Let $\Omega$ be a smoothly bounded domain in $\mathbb{C}^{n}$. A non
negative measurable function $\nu$ on $\Omega$ is said to be an admissible
weight if the space of holomorphic functions square integrable for
the measure $\nu d\lambda$ ($d\lambda$ being the Lebesgue measure)
is a closed subspace of the Hilbert space $L^{2}\left(\nu d\lambda\right)$
of square integrable functions on $\Omega$ (see, for example, \cite{Pas90})
. In complex analysis, $\nu$ being admissible, the regularity of
the Bergman projection associated to $\nu d\lambda$ (i.e. the orthogonal
projection of $L^{2}\left(\nu d\lambda\right)$ onto the subspace
of holomorphic functions) is a fundamental question. It has been intensively
studied when $\nu\equiv1$ and specially when $\Omega$ is pseudoconvex.

If $\eta$ is a smooth strictly positive function on $\overline{\Omega}$
it is well known that the regularity properties of the Bergman projections of
the Hilbert spaces $L^{2}\left(\eta\nu d\lambda\right)$ and $L^{2}\left(\nu d\lambda\right)$
can be very different. For example in \cite{Koh72} J. J. Kohn proved
that if $\Omega$ is pseudoconvex, for any integer $k$ there exists
$t>0$ such that the Bergman projection of $L^{2}\left(e^{-t\left|z\right|^{2}}d\lambda\right)$
maps the Sobolev space $L_{k}^{2}(\Omega)$ into itself, and, in
\cite{Christ96} M. Christ showed that there exists a smoothly bounded
pseudoconvex domain such that the Bergman projection of $L^{2}(\Omega)=L^{2}\left(e^{t\left|z\right|^{2}}e^{-t\left|z\right|^{2}}d\lambda\right)$
is not $L^{2}$-Sobolev regular.

In this paper we show that some of the (weighted) estimates obtained
in \cite{CDM} for pseudoconvex domains of finite type remain true
when the weight is multiplied by a function which is smooth and strictly
positive in $\overline{\Omega}$. This shows that the corresponding
estimates obtained in \cite{CDM} for the Bergman projection of $L^{2}\left(\left(-\rho_{0}\right)^{r}d\lambda\right)$,
where $\rho_{0}$ is a special defining function of $\Omega$ and
$r$ a non negative rational number, are valid for the Bergman projection
of $L^{2}\left(\left(-\rho\right)^{r}d\lambda\right)$ where $\rho$
is any defining function of the domain. Moreover, we show that these Bergman
projections satisfy a stronger directional Sobolev estimate.

\section{\label{sec:Notations-and-main-results}Notations and main results}

Throughout all the paper $d\lambda$ denotes the Lebesgue measure.
Let $D$ be a smoothly bounded open set in $\mathbb{C}^{l}$. Recall
that $d$ is said to be a defining function of $D$ if it is a real
function in $\mathcal{C}^{\infty}\left(\mathbb{C}^{l}\right)$ such
that $D=\left\{ \zeta\in\mathbb{C}^{l}\mbox{ s. t. }d(\zeta)<0\right\} $
and $\nabla d$ does not vanish on $\partial D$.

Let $\nu$ be an admissible weight on $D$.

For $1\leq p<+\infty$ we denote by $L^{p}\left(D,\nu d\lambda\right)$
the $L^{p}$ space for the measure $\nu d\lambda$. When $\nu\equiv1$
we write, as usual, $L^{p}(D)$.

We denote by $P_{\nu}^{D}$ the orthogonal projection of the Hilbert
space $L^{2}\left(D,\nu d\lambda\right)$ (i. e. for the scalar
product $\left\langle f,g\right\rangle =\int_{D}f\overline{g}\nu d\lambda$)
onto the closed subspace of holomorphic functions. If $\nu\equiv1$
we simply write $P^{D}$. In this paper, $P_{\nu}^{D}$ is called
the (weighted) \emph{Bergman projection of} $L^{2}\left(D,\nu d\lambda\right)$.

For $k\in\mathbb{N}$ and $1<p<+\infty$, we define the weighted Sobolev
space $L_{k}^{p}\left(D,\nu d\lambda\right)$ by

\begin{multline*}
L_{k}^{p}\left(D,\nu d\lambda\right)=\bigg\{ u\in L^{p}\left(D,\nu d\lambda\right)\mbox{ such that }\\
\left.\left\Vert u\right\Vert _{L_{k}^{p}\left(D,\nu d\lambda\right)}^{p}=\sum_{\left|\alpha\right|\leq k}\left\Vert D^{\alpha}u\right\Vert _{L^{p}\left(D,\nu d\lambda\right)}^{p}<+\infty\right\} .
\end{multline*}
If $\nu\equiv1$ this space is the classical Sobolev space $L_{k}^{p}\left(D\right)$.

Let $d$ be a smooth defining function of $D$. We denote by $T_{d}$
the vector field
\[
T_{d}=\sum_{i}\frac{\partial d}{\partial\overline{z_{i}}}\frac{\partial}{\partial z_{i}}-\frac{\partial d}{\partial z_{i}}\frac{\partial}{\partial\overline{z_{i}}}.
\]
Thus $T_{d}$ is a vector field tangent to $d$ (i.e. $T_{d}d\equiv0$)
which is transverse to the complex tangent space to $d$ near the
boundary of $D$.

Following a terminology introduced in \cite{HMS14} a vector field
$T$ with coefficients in $\mathcal{C}^{\infty}(\overline{D})$ is
said to be \emph{tangential and complex transversal} to $\partial\Omega$
if it can be written $T=aT_{d}+L$ where $a\in\mathcal{C}^{\infty}(\overline{D})$
is nowhere vanishing on $\partial D$ and $L=L_{1}+\overline{L_{2}}$
where $L_{1}$ and $L_{2}$ are $\left(1,0\right)$-type vector fields
tangential to $\partial D$. Note that this definition is independent
of the choice of the defining function $d$ (see the beginning of
\secref{Proof-of-General-Sobolev-L2}). Then, for all non negative integer $k$, and $1<p<+\infty$,
we denote by $L_{k,T}^{p}\left(D,\nu d\lambda\right)$ the weighted
directional Sobolev space
\begin{multline*}
L_{k,T}^{p}\left(D,\nu d\lambda\right)=\bigg\{ u\in L^{p}\left(D,\nu d\lambda\right)\mbox{ such that }\\
\left.\left\Vert u\right\Vert _{L_{k,T}^{p}\left(D,\nu d\lambda\right)}^{p}=\sum_{l\leq k}\left\Vert T^{l}u\right\Vert _{L^{p}\left(D,\nu d\lambda\right)}^{p}<+\infty\right\} .
\end{multline*}

Our first result extends Theorem 2.2 of \cite{CDM} and, for finite
type domains, Theorem 1.1 obtained by A.-K. Herbig, J. D. McNeal and
E. J. Straube in \cite{HMS14} for the standard Bergman projection:
\begin{stthm}
\label{thm:General_Sobolev_estimates}Let $\Omega$ be a smooth bounded
pseudoconvex domain of finite type in $\mathbb{C}^{n}$. Let $\rho$
be a smooth defining function of $\Omega$. Let $r\in\mathbb{Q}_{+}$,
be a non negative rational number and $\eta\in\mathcal{C}^{\infty}(\overline{\Omega})$,
strictly positive. Let $T$ be a $\mathcal{C}^{\infty}(\overline{\Omega})$
vector field tangential and complex transversal to $\partial\Omega$.
Define $\omega=\eta\left(-\rho\right)^{r}$. Then, for any integer
$k$, $P_{\omega}^{\Omega}$ maps continuously the weighted directional
Sobolev space $L_{k,T}^{2}\left(\Omega,\omega d\lambda\right)$ into
$L_{k}^{2}\left(\Omega,\omega d\lambda\right)$.
\end{stthm}
Note that $r$ is allowed to be $0$.
\begin{cor*}
In the conditions stated in the theorem, $P_{\omega}^{\Omega}$ maps
continuously $$\cap_{k\in\mathbb{N}}L_{k,T}^{2}\left(\Omega,\omega_{0}d\lambda\right)$$
into $\mathcal{C}^{\infty}(\overline{\Omega})$.\end{cor*}
Our second result is inspired by Theorem 1.10 of \cite{HMS14}:
\begin{stthm}
\label{thm:Cinfini-estimate-conj-holo}Let $\Omega$, $\eta$ and
$\omega$ as in \thmref{General_Sobolev_estimates}. Let $f\in L^{2}\left(\Omega,\omega d\lambda\right)$
such that $\overline{f}$ is holomorphic and let $h\in\mathcal{C}^{\infty}(\overline{\Omega})$.
Then $P_{\omega}^{\Omega}(fh)\in\mathcal{C}^{\infty}(\overline{\Omega})$.
\end{stthm}
The proofs are done in \secref{Proof-of-General-Sobolev-L2}.

\smallskip{}

Our other results are partial generalizations of Theorem 2.1 of \cite{CDM}
for domains in $\mathbb{C}^{2}$ and for convex domains. As the results
for convex domains are not general we state them separately and we
will only indicate the articulations of the proof at the end of \secref{Proof-of-theorem-C2-convex}.

The first result extends the results obtained by A. Bonami and S. Grellier in \cite{BG95}:
\begin{stthm}
\label{thm:Estimates-C2-rank-Levi}Let $\Omega$, $\eta$ and $\omega$
as in \thmref{General_Sobolev_estimates}. Assume moreover that, at
every point of $\partial\Omega$, the rank of the Levi form is $\geq n-2$.
Then:
\begin{enumerate}
\item For $1<p<+\infty$ and $k\in\mathbb{N}$, $P_{\omega}^{\Omega}$ maps
continuously the Sobolev space $L_{k}^{p}\left(\Omega,\omega d\lambda\right)$
into itself;
\item For $\alpha<1$, $P_{\omega}^{\Omega}$
maps continuously the Lipshitz space $\Lambda_{\alpha}\left(\Omega\right)$
into itself.
\end{enumerate}
\end{stthm}

The results for convex domains are identical but under an additional
condition on the existence of a special defining function. To state
it we recall a terminology introduced in Section 2.4 of \cite{CDM}:

If $g$ is a real or complex valued smooth function defined in a neighborhood
of the origin in $\mathbb{R}^{d}$, we call the \emph{order} of $g$
at the origin the integer $\mbox{ord}_{0}(g)$ defined by $\mbox{ord}_{0}(g)=\infty$
if $g^{(\alpha)}(0)=0$ for all multi-index $\alpha\in\mathbb{N}^{d}$
and  
\begin{multline*}
\mbox{ord}_{0}(g)=\min\Big\{ k\in\mathbb{N}\mbox{ such that there exists }\alpha\in\mathbb{N}^{d},\,\\
\left.\left|\alpha\right|=\sum\alpha_{i}=k\mbox{ such that }g^{(\alpha)}(0)\neq0\right\} 
\end{multline*}
otherwise. If $\psi$ is a smooth function defined in a neighborhood
of the origin in $\mathbb{C}^{m}$, then, for all function $\varphi$
from the unit disc of the complex plane into $\mathbb{C}^{m}$ such
that $\varphi(0)=0$, $\psi\circ\varphi$ is smooth in a neighborhood
of the origin in $\mathbb{C}$. Then we call the \emph{type} of $\psi$
at the origin the supremum of $\frac{\mbox{ord}_{0}\left(\psi\circ\varphi\right)}{\mbox{ord}_{0}\left(\varphi\right)}$,
taken over all non zero holomorphic function $\varphi$ from the unit
disc of the complex plane into $\mathbb{C}^{m}$ such that $\varphi(0)=0$.
If this supremum is finite, we say that $\psi$ is of \emph{finite
type} at the origin and we denote this supremum by $\mbox{typ}_{0}(\psi)$.
Moreover, if $\vartheta$ is a smooth function defined in a neighborhood
of a point $z_{0}\in\mathbb{C}^{m}$, the type $\mbox{typ}_{z_{0}}\left(\vartheta\right)$
of $\vartheta$ at $z_{0}$ is $\mbox{typ}_{0}\left(\vartheta_{k}\right)$
where $\vartheta_{k}(z)=\vartheta\left(z_{0}+z\right)$ and we say
that $\vartheta$ is of finite type at $z_{0}$ if $\mbox{typ}_{z_{0}}\left(\vartheta\right)<+\infty$.
If $\vartheta$ is defined on a neighborhood of a set $S$ we say
that $\vartheta$ is of finite type on $S$ if $\sup_{z\in S}\mbox{typ}_{z}\left(\vartheta\right)<+\infty$.
\begin{spprop}
\label{prop:Estimates_for_Convex_domains}Let $\Omega$, $\eta$ and
$\omega$ as in \thmref{General_Sobolev_estimates}. Assume that $\Omega$
is convex and admits a defining function which is smooth, convex and
of finite type in $\overline{\Omega}$. Then:
\begin{enumerate}
\item For $1<p<+\infty$ and $k\in\mathbb{N}$, $P_{\omega}^{\Omega}$ maps
continuously the Sobolev space $L_{k}^{p}\left(\Omega,\omega d\lambda\right)$
into itself;
\item For $\alpha<1$, $P_{\omega}^{\Omega}$
maps continuously the Lipshitz space $\Lambda_{\alpha}\left(\Omega\right)$
into itself.
\end{enumerate}
\end{spprop}
\begin{rem*}
The defining function choosen in \cite{CDM} for a general convex
domain of finite type is smooth convex of finite type everywhere except
at one point.
\end{rem*}

\thmref{Estimates-C2-rank-Levi} is proved in \secref{Proof-of-theorem-C2-convex}
as a special case of a stronger directional $L_{k}^{p}$ estimate
(\thmref{reformulation-C-2}).

\smallskip{}

The general scheme of the proofs of these results is as follows. Recall
that in \cite{CDM} we obtain the estimates in the above theorems
for the projections $P_{\omega_{0}}^{\Omega}$ where $\omega_{0}=\left(-\rho_{0}\right)^{r}$,
$\rho_{0}$ being the following special defining function of $\Omega$:
\begin{itemize}
\item For \thmref{General_Sobolev_estimates} and for \thmref{Estimates-C2-rank-Levi},
using a celebrated theorem of K. Diederich \& J. E. Forn\ae{}ss (\cite[Theorem 1]{DF77-Strict-Psh-Exhau-Func-Inventiones}),
$\rho_{0}$ is chosen so that there exists $t\in\left]0,1\right[$
such that $-\left(-\rho_{0}\right)^{t}$ is strictly plurisubharmonic
in $\Omega$.
\item If $\Omega$ is convex $\rho_{0}$ is assumed to be convex and of
finite type in $\overline{\Omega}$ (hypothesis of \propref{Estimates_for_Convex_domains}).
\end{itemize}

Then we obtain the results for $P_{\omega}^{\Omega}$ comparing
$P_{\omega}^{\Omega}$ and $P_{\omega_{0}}^{\Omega}$ as explained
in the next section. The restriction imposed to $\Omega$ in \propref{Estimates_for_Convex_domains}
comes from the fact that this comparison uses estimates with gain for
solutions of the $\overline{\partial}$-equation in a domain $\widetilde{\Omega}$
in $\mathbb{C}^{n+m}$, $n+m\geq3$, which are known only under strong
hypothesis on $\widetilde{\Omega}$.
\medskip{}
As $\rho$ and $\rho_{0}$ are two smooth defining functions of $\Omega$
there exists a function $\varphi$ smooth and strictly positive on
$\overline{\Omega}$ such that $\rho=\varphi\rho_{0}$. Then, there
exists a function $\eta_{1}\in\mathcal{C}^{\infty}(\overline{\Omega})$,
$\eta_{1}>0$, such that $\omega=\eta_{1}\left(-\rho_{0}\right)^{r}$.
\medskip{}

\emph{Thus, from now on, $\rho_{0}$ and $\omega_{0}$ are fixed as
above and, to simplify the notations, we write $\omega=\eta\left(-\rho_{0}\right)^{r}$
where $\eta$ is a strictly positive function in $\mathcal{C}^{1}(\overline{\Omega})$.}

\section{Comparing \texorpdfstring{$P_{\omega}^{\Omega}$}{P\textomega}
and \texorpdfstring{$P_{\omega_{0}}^{\Omega}$}{P\textomega\textzeroinferior}}

This comparison is based on the following simple formula:
\begin{spprop}
With the previous notations for $D$ and $P_{\nu}^{D}$, let $\eta$
be a strictly positive function in $\mathcal{C}^{\infty}\left(\overline{D}\right)$
(so that $\eta\nu$ is an admissible weight). Let $L_{\left(0,1\right)}^{2}\left(D,\nu d\lambda\right)$
be the space of $\left(0,1\right)$-forms with coefficients in $L^{2}\left(D,\nu d\lambda\right)$.
If there exists a continuous linear operator $A_{\nu}$ from $L_{\left(0,1\right)}^{2}\left(D,\nu d\lambda\right)\cap\ker\overline{\partial}$
into $L^{2}\left(D,\nu d\lambda\right)$ such that, for $f\in L_{\left(0,1\right)}^{2}\left(D,\nu d\lambda\right)\cap\ker\overline{\partial}$,
$A_{\nu}(f)$ is orthogonal to holomorphic functions in $L^{2}\left(D,\nu d\lambda\right)$
and $\overline{\partial}A_{\nu}(f)=f$ then, for all $u\in L^{2}\left(D,\nu d\lambda\right)$
we have
\[
\eta P_{\eta\nu}^{D}(u)=P_{\nu}^{D}\left(\eta u\right)+A_{\nu}\left(P_{\eta\nu}^{D}(u)\overline{\partial}\eta\right).
\]
\end{spprop}
\begin{proof}
This is almost immediate: from the second hypothesis on $A_{\nu}$
both sides of the formula have same $\overline{\partial}$, and, from
the first hypothesis, both sides have same scalar product, in $L^{2}\left(D,\nu d\lambda\right)$,
against holomorphic functions.
\end{proof}

We use this formula in the context developed in \cite{CDM}.

For $h(w)=\left|w\right|^{2q}$, $w\in\mathbb{C}^{m}$, $r=\nicefrac{q}{m}$
or $h(w)=\sum\left|w_{i}\right|^{2q_{i}}$, $w_{i}\in\mathbb{C}$,
$r=\sum\nicefrac{1}{q_{i}}$ (c.f. \cite{CDM}), $\rho_{0}$
and $\omega_{0}$ as introduced in the preceding section, we consider the domain in
$\mathbb{C}^{n+m}$ defined by
\[
\widetilde{\Omega}=\left\{ \left(z,w\right)\in\mathbb{C}^{n}\times\mathbb{C}^{m}\mbox{, s. t. }r(z,w)=\rho_{0}(z)+h(w)<0\right\} .
\]
Then (c.f. \cite{CDM}) $\widetilde{\Omega}$ is smooth, bounded and
pseudoconvex. Therefore the $\overline{\partial}$-Neumann operator
$\mathcal{N}_{\widetilde{\Omega}}$ is well defined. Let us introduce
two notations:
\begin{itemize}
\item If $u\in L^{p}\left(\Omega,\omega_{0}d\lambda\right)$, $1\leq p<+\infty$,
we denote by $I(u)$ the function, belonging to $L^{p}(\widetilde{\Omega})$,
defined by $I(u)\left(z,w\right)=u(z)$ (the fact that $I(u)\in L^{p}(\widetilde{\Omega})$
follows Fubini's theorem). We extend this notations to forms $f=\sum f_{i}d\overline{z}_{i}$
in $L_{(0,1)}^{p}\left(\Omega,\omega_{0}d\lambda\right)$ by $I(f)=\sum I\left(f_{i}\right)d\overline{z}_{i}$
(so that $I(f)\in L_{(0,1)}^{p}(\widetilde{\Omega})$ and, $I(f)$
is $\overline{\partial}$-closed if $f$ is so).
\item If $v\in L^{p}(\widetilde{\Omega})$, $1\leq p<+\infty$, is holomorphic
in $w$ we denote by $R(v)$ the function, belonging to $L^{p}\left(\Omega,\omega_{0}d\lambda\right)$
(by the mean value property applied to the subharmonic function $w\mapsto\left|v(z,w)\right|^{p}$),
defined by $R(v)(z)=v\left(z,0\right)$.
\end{itemize}
Then:
\begin{spprop}
\label{prop:Relation-P-omega_P-omega-0}For any function $u\in L^{2}\left(\Omega,\omega d\lambda\right)$,
we have
\begin{equation}
\eta P_{\omega}^{\Omega}(u)=P_{\omega_{0}}^{\Omega}\left(\eta u\right)+R\circ\left(\overline{\partial}^{*}\mathcal{N}_{\widetilde{\Omega}}\right)\circ I\left(P_{\omega}^{\Omega}(u)\overline{\partial}\left(\eta\right)\right).\label{eq:comparison-P-omega-P-omega0}
\end{equation}
\end{spprop}
\begin{proof}
By the preceding proposition, it suffices to note that the operator
$R\circ\left(\overline{\partial}^{*}\mathcal{N}_{\widetilde{\Omega}}\right)\circ I$
is continuous from $L_{(0,1)}^{2}\left(\Omega,\omega_{0}d\lambda\right)\cap\ker\overline{\partial}$
into $L^{2}\left(\Omega,\omega_{0}d\lambda\right)$, solves the $\overline{\partial}$-equation
and gives the solution which is orthogonal to holomorphic functions
in that space. But if $f\in L_{(0,1)}^{2}\left(\Omega,\omega_{0}d\lambda\right)\cap\ker\overline{\partial}$
then, by Fubini's theorem, $I(f)\in L_{(0,1)}^{2}(\widetilde{\Omega})\cap\ker\overline{\partial}$,
and $\left(\overline{\partial}^{*}\mathcal{N}_{\widetilde{\Omega}}\right)\circ I(f)$
is the solution of $\overline{\partial}u=I(f)$ which is orthogonal
to holomorphic functions in $L^{2}(\widetilde{\Omega})$ and satisfies
\[
\left\Vert \left(\overline{\partial}^{*}\mathcal{N}_{\widetilde{\Omega}}\right)\circ I(f)\right\Vert _{L^{2}\left(\widetilde{\Omega}\right)}\lesssim\left\Vert I(f)\right\Vert _{L_{(0,1)}^{2}\left(\widetilde{\Omega}\right)}=C\left\Vert f\right\Vert _{L_{(0,1)}^{2}\left(\Omega,\omega_{0}d\lambda\right)}
\]
(recall that $\widetilde{\Omega}$ is pseudoconvex and that the volume of $\left\{ h(w)<-\rho_{0}(z)\right\}$
is equal to $C\omega_{0}(z)$). As $I(f)$ is independent
of the variable $w$, $\left(\overline{\partial}^{*}\mathcal{N}_{\widetilde{\Omega}}\right)\circ I(f)$
is holomorphic in $w$ and 
\[
\overline{\partial}_{z}\left(\left(\overline{\partial}^{*}\mathcal{N}_{\widetilde{\Omega}}\right)\circ I(f)\right)(z,0)=f(z)
\]
so $\overline{\partial}\left(R\circ\left(\overline{\partial}^{*}\mathcal{N}_{\widetilde{\Omega}}\right)\circ I(f)\right)=f$,
and, by the mean value property (applied to the subharmonic function
$w\mapsto\left|\left(\overline{\partial}^{*}\mathcal{N}_{\widetilde{\Omega}}\right)\circ I(f)(z,w)\right|^{2}$),
\[
\left\Vert R\circ\left(\overline{\partial}^{*}\mathcal{N}_{\widetilde{\Omega}}\right)\circ I(f)\right\Vert _{L^{2}\left(\Omega,\omega_{0}d\lambda\right)}\leq C\left\Vert \left(\overline{\partial}^{*}\mathcal{N}_{\widetilde{\Omega}}\right)\circ I(f)\right\Vert _{L^{2}\left(\widetilde{\Omega}\right)}\lesssim\left\Vert f\right\Vert _{L_{(0,1)}^{2}\left(\Omega,\omega_{0}d\lambda\right)}.
\]

Moreover, if $g$ is a holomorphic function in $L^{2}\left(\Omega,\omega_{0}d\lambda\right)$,
by the mean value property,
\begin{eqnarray*}
\int_{\Omega}R\circ\left(\overline{\partial}^{*}\mathcal{N}_{\widetilde{\Omega}}\right)\circ I(f)\overline{g}\omega_{0}d\lambda & = & \int_{\Omega}\left(\overline{\partial}^{*}\mathcal{N}_{\widetilde{\Omega}}\right)\circ I(f)(z,0)\overline{g}(z)\omega_{0}(z)d\lambda(z)\\
 & = & C\int_{\Omega}\left(\int_{\left\{ h(w)<-\rho_{0}(z)\right\} }\left(\overline{\partial}^{*}\mathcal{N}_{\widetilde{\Omega}}\right)\right.\\
 & & \quad\qquad\qquad\qquad\circ I(f)(z,w)d\lambda(w)\bigg)\overline{g}(z)d\lambda(z)\\
 & = & C\int_{\widetilde{\Omega}}\left(\overline{\partial}^{*}\mathcal{N}_{\widetilde{\Omega}}\right)\circ I(f)(z,w)\overline{g}(z)d\lambda(z,w)=0.
\end{eqnarray*}

\end{proof}

An immediate density argument shows that:
\begin{cor*}
Let $p\in\left]1,+\infty\right[$. Assume that the following properties
are satisfied:
\begin{itemize}
\item $P_{\omega}^{\Omega}$ and $P_{\omega_{0}}^{\Omega}$ map continuously
$L^{p}\left(\Omega,\omega d\lambda\right)$ into itself;
\item $\overline{\partial}\mathcal{N}_{\widetilde{\Omega}}$ maps continuously
$L^{p}(\widetilde{\Omega})$ into itself.
\end{itemize}
Then \eqref{comparison-P-omega-P-omega0} is valid for any function
$u\in L^{p}\left(\Omega,\omega d\lambda\right)$.

\medskip{}

\end{cor*}
In the proofs of the theorems we need to use weighted Sobolev spaces
$L_{s}^{p}\left(D,\nu d\lambda\right)$ defined for all $s\geq0$
and some directional Sobolev spaces.

It is well known that for $s\in\left[k,k+1\right]$, $k\in\mathbb{N}$,
the fractional Sobolev space $L_{s}^{p}\left(D\right)$ is obtained
using the complex interpolation method between $L_{k}^{p}\left(D\right)$
and $L_{k+1}^{p}\left(D\right)$ (see, for example, \cite{Tri78}).
By analogy, we extend the definition of the weighted Sobolev spaces
$L_{k}^{p}\left(D,\nu d\lambda\right)$ to any index $s\geq0$ using
the complex interpolation method:
\[
L_{s}^{p}\left(D,\nu d\lambda\right)=\left[L_{k}^{p}\left(D,\nu d\lambda\right),L_{k+1}^{p}\left(D,\nu d\lambda\right)\right]_{s-k},\mbox{ if }s\in\left[k,k+1\right].
\]
Note that if $\nu_{1}$ and $\nu_{2}$ are two admissible weights
such that $\nu_{2}=\eta\nu_{1}$ with $\eta$ a strictly positive
function in $\mathcal{C}^{\left[s\right]+1}\left(\overline{\Omega}\right)$
then the Banach spaces $L_{s}^{p}\left(D,\nu_{1}d\lambda\right)$
and $L_{s}^{p}\left(D,\nu_{2}d\lambda\right)$ are identical.

\smallskip{}

For all $s\geq0$, we extend the definition of $L_{k,T}^{p}\left(D,\nu d\lambda\right)$
to $L_{s,T}^{p}\left(D,\nu d\lambda\right)=\left[L_{k,T}^{p}(D,\nu d\lambda),L_{k+1,T}^{p}(D,\nu d\lambda)\right]_{s-k}$,
$k\leq s\leq k+1$ by complex interpolation between two consecutive
integers. Clearly, the spaces $L_{s,T}^{2}\left(D,\nu d\lambda\right)$
are Hilbert spaces and $L_{s,T}^{p}\left(D,\nu d\lambda\right)$
are Banach spaces. When $\nu\equiv1$ we denote this space $L_{s,T}^{p}(D)$.

Note that, $r=\rho_{0}+h$ being the defining function of $\widetilde{\Omega}$,
we have $T_{r}=T_{\rho_{0}}+T_{h}$, with $T_{h}=\left|w\right|^{2q-2}\sum\left(w_{i}\frac{\partial}{\partial w_{i}}-\overline{w_{i}}\frac{\partial}{\partial\overline{w_{i}}}\right)$
when $h(w)=\left|w\right|^{2q}$, $w\in\mathbb{C}^{m}$ and $T_{h}=\sum_{i=1}^{m}\left|w_{i}\right|^{2q_{i}-2}\left(w_{i}\frac{\partial}{\partial w_{i}}-\overline{w_{i}}\frac{\partial}{\partial\overline{w_{i}}}\right)$
when $h(w)=\sum_{i=1}^{m}\left|w_{i}\right|^{2q_{i}}$, $w_{i}\in\mathbb{C}$.
\begin{srrem}
The spaces $L_{s,T}^{p}\left(D,\nu d\lambda\right)$ depend on
the choice of the vector field $T$
(see Section 5 of \cite{Herbig-McNeal_Smoo-Berg_2012}).
\end{srrem}
\medskip{}

We now state some elementary properties of the operators $I$ and
$R$ introduced before and related to these Sobolev spaces. It is
convenient to introduce other spaces: for $1<p<+\infty$ and
$s\geq0$, let 
\[
L_{s}^{p}(\widetilde{\Omega})\cap\ker\overline{\partial}_{w}=\left\{ u(z,w)\in L_{s}^{p}(\widetilde{\Omega})\mbox{ such that }\frac{\partial u}{\partial\overline{w_{i}}}\equiv0,\,1\leq i\leq m\right\} .
\]

\begin{sllem}
\label{lem:I_and_R_global}With the previous notations and for $1<p<+\infty$,
we have:
\begin{enumerate}
\item \label{Reg_I_W-s}For all $s\geq0$, $I$ maps continuously $L_{s}^{p}\left(\Omega,\omega_{0}d\lambda\right)$
into $L_{s}^{p}(\widetilde{\Omega})$.
\item \label{Reg_R_W-k}For all non negative integer $k$, $R$ maps continuously
$L_{k}^{p}(\widetilde{\Omega})\cap\ker\overline{\partial}_{w}$ into
$L_{k}^{p}\left(\Omega,\omega_{0}d\lambda\right)$.
\end{enumerate}
\end{sllem}
\begin{proof}
As $D_{z}^{\alpha}I(h)=I\left(D^{\alpha}h\right)$ for any derivative
$D^{\alpha}$, Fubini's Theorem implies $\left\Vert D_{z}^{\alpha}I(h)\right\Vert _{L^{p}(\widetilde{\Omega})}^{p}=C\left\Vert D^{\alpha}h\right\Vert _{L^{p}\left(\Omega,\omega_{0}d\lambda\right)}^{p}$
and (\ref{Reg_I_W-s}) follows for $s\in\mathbb{N}$ and for all $s$ by the interpolation theorem.

The second point of the lemma is also very simple. If $u\in L_{k}^{p}(\widetilde{\Omega})\cap\ker\overline{\partial}_{w}$,
then, for all derivative $D^{\alpha}$, $D^{\alpha}\left(Ru\right)(z)=D_{z}^{\alpha}u(z,0)$,
and $w\mapsto\left|D_{z}^{\alpha}u(z,w)\right|^{p}$ is subharmonic.
Therefore the mean value property gives
\[
C\left|D^{\alpha}(Ru)(z)\right|^{p}\omega_{0}(z)\leq\int_{\left\{ h(w)<-\rho_{0}(z)\right\} }\left|D_{z}^{\alpha}u(z,w)\right|^{p}d\lambda(w).
\]
Integrating this inequality over $\Omega$ finishes the proof.
\end{proof}
\medskip{}

In Sections \ref{sec:Proof-of-General-Sobolev-L2} and \ref{sec:Proof-of-theorem-C2-convex}
we will need estimates for $R$ on the spaces $L_{s}^{p}(\widetilde{\Omega})\cap\ker\overline{\partial}_{w}$
for all $s\geq0$. Unfortunately the two spaces
$L_{s}^{p}(\widetilde{\Omega})\cap\ker\overline{\partial}_{w}=\left[L_{k}^{p}(\widetilde{\Omega}),L_{k+1}^{p}(\widetilde{\Omega})\right]_{s-k}\cap\ker\overline{\partial}_{w}$
and
$\left[L_{k}^{p}(\widetilde{\Omega})\cap\ker\overline{\partial}_{w},L_{k+1}^{p}(\widetilde{\Omega})\cap\ker\overline{\partial}_{w}\right]_{s-k}$
may be different. This difficulty is circumvented by the following
lemma:
\begin{sllem}
\label{lem:Lemma_I_R}With the previous notations and for $1<p<+\infty$,
we have:
\begin{enumerate}
\item \label{Reg_I_W-s-T}For all $s\geq0$, $I$ maps continuously $L_{s,T_{\rho_{0}}}^{p}\left(\Omega,\omega_{0}d\lambda\right)$
into $L_{s,T_{r}}^{p}(\widetilde{\Omega})$.
\item \label{Reg_R_W-s-T}For all $s\geq0$, $R$ maps continuously $L_{s}^{p}(\widetilde{\Omega})\cap\ker\overline{\partial}_{w}$
into $L_{s,T_{\rho_{0}}}^{p}\left(\Omega,\omega_{0}d\lambda\right)$.
\end{enumerate}
\end{sllem}
\begin{proof}
As $T_{r}^{l}\left(I(h)\right)=I\left(T_{\rho_{0}}^{l}(h)\right)$
Fubini's Theorem gives (\ref{Reg_I_W-s-T}) when $s$ is an integer.
Therefore (\ref{Reg_I_W-s-T}) follows by interpolation.

To see the second point of the lemma, let us denote by $M_{0}u$ the
mean with respect to the variable $w$ of a function $u$ in $L^{p}(\widetilde{\Omega})$
\[
M_{0}u(z)=\frac{1}{C\omega_{0}(z)}\int_{\left\{ h(w)<-\rho_{0}(z)\right\} }u(z,w)d\lambda(w).
\]

As $T_{\rho_{0}}$ is tangent to $\rho_{0}$, we have $T_{\rho_{0}}\left(\omega_{0}\right)\equiv T_{\rho_{0}}\left(\rho_{0}\right)\equiv0$,
and, for all integer $l$ we get
\[
T_{\rho_{0}}^{l}M_{0}u(z)=\frac{1}{C\omega_{0}(z)}\int_{\left\{ h(w)<-\rho_{0}(z)\right\} }T_{\rho_{0}}^{l}u(z,w)d\lambda(w).
\]

Then, by H\"{o}lder inequality we have
\[
C\left|T_{\rho_{0}}^{l}M_{0}u(z)\right|^{p}\omega_{0}(z)\leq\int_{\left\{ h(w)<-\rho_{0}(z)\right\} }\left|T_{\rho_{0}}^{l}u(z,w)\right|^{p}d\lambda
\]
and, integrating this inequality over $\Omega$, we get that $M_{0}$
maps continuously $L_{k}^{p}(\widetilde{\Omega})$ into $L_{k,T_{\rho_{0}}}^{p}\left(\Omega,\omega_{0}d\lambda\right)$.
Therefore, by the interpolation theorem, $M_{0}$ maps continuously
$L_{s}^{p}(\widetilde{\Omega})$ into $L_{s,T_{\rho_{0}}}^{p}\left(\Omega,\omega_{0}d\lambda\right)$
for all $s\geq0$.

This proves (\ref{Reg_R_W-s-T}) of the lemma because, by the mean
value property for holomorphic functions, $M_{0}u=Ru$ when $u\in L_{s}^{p}(\widetilde{\Omega})\cap\ker\overline{\partial}_{w}$.
\end{proof}

\section{\label{sec:Proof-of-General-Sobolev-L2}Proof of Theorems \texorpdfstring{\ref{thm:General_Sobolev_estimates}}{2.1}
and \texorpdfstring{\ref{thm:Cinfini-estimate-conj-holo}}{2.2}}

For convenience, we extend the notation of the vector field $T_{d}$
given at the beginning of \secref{Notations-and-main-results} denoting
by $T_{\psi}$ the vector field
\[
T_{\psi}=\sum\frac{\partial\psi}{\partial\overline{z}_{j}}\frac{\partial}{\partial z_{j}}-\frac{\partial\psi}{\partial z_{j}}\frac{\partial}{\partial\overline{z}_{j}}
\]
where $\psi$ is any function in $\mathcal{C}^{1}(\overline{D})$.

If $T=aT_{\rho}+L$ is the vector field given in \thmref{General_Sobolev_estimates}
then (writing $\rho=\varphi\rho_{0}$) we have $T=a\varphi T_{\rho_{0}}+\left(\rho_{0}T_{\varphi}+L\right)=a\varphi T_{\rho_{0}}+L'$
with $\varphi>0$ on $\overline{\Omega}$ and $L'=L'_{1}+\overline{L'_{2}}$
where $L'_{1}$ and $L'_{2}$ are $\left(1,0\right)$-type vector
fields tangential to $\partial\Omega$. Moreover, writing $a=a'+b$
where $a'$ is nowhere vanishing on $\overline{\Omega}$ and $b$
identically $0$ in a neighborhood of $\partial\Omega$, we get $T=a'\varphi T_{\rho_{0}}+L''$
with $L''=L''_{1}+L''_{2}$ where $L''_{1}$ and $L''_{2}$ are $\left(1,0\right)$-type
vector fields tangential to $\partial\Omega$ and $a'\varphi$ is
nowhere vanishing on $\overline{\Omega}$.

\bigskip{}

We now prove the following reformulation of \thmref{General_Sobolev_estimates}:
\begin{stthm}
\label{thm:Reformulation-General-Sobolev}Let $\Omega$ be as in \thmref{General_Sobolev_estimates}.
Let $k$ be a non negative integer. Let $\rho_{0}$, $\omega_{0}$
and $\omega$ be as at the end of \secref{Notations-and-main-results} with
$\eta\in\mathcal{C}^{k+1}(\overline{\Omega})$. Let $\varphi\in\mathcal{C}^{\infty}(\overline{\Omega})$
a function which is nowhere vanishing on $\overline{\Omega}$ and
let $T=\varphi T_{\rho_{0}}+L$ with $L=L_{1}+L_{2}$ where $L_{1}$
and $L_{2}$ are $\mathcal{C}^{\infty}(\overline{\Omega})$ vector
fields of type $\left(1,0\right)$ tangential to $\partial\Omega$.
Then for $s\in\left[0,k\right]$ the weighted Bergman projection $P_{\omega}^{\Omega}$
maps continuously the directional weighted Sobolev space
$L_{s,T}^{2}\left(\Omega,\omega_{0}d\lambda\right)$
into $L_{s}^{2}\left(\Omega,\omega_{0}d\lambda\right)$.
\end{stthm}
\begin{proof}
With the notations of the end of the preceding section, we choose
$h(w)=\left|w\right|^{2q}$, $w\in\mathbb{C}^{m}$ with $r=\nicefrac{m}{q}$.
Then, by results of \cite{CDM}, $\widetilde{\Omega}$ is a smoothly
bounded pseudoconvex domain in $\mathbb{C}^{n+m}$ of finite type.
First we note that the estimate of the theorem for $P_{\omega_{0}}^{\Omega}$
is a consequence of a theorem of A.-K. Herbig, J. D. McNeal and E. Straube:
\begin{sllem}
\label{lem:P-omega-s_W-T-s}The Bergman projection $P_{\omega_{0}}^{\Omega}$
maps continuously the directional space $L_{s,T}^{2}\left(\Omega,\omega_{0}d\lambda\right)$
into $L_{s}^{2}\left(\Omega,\omega_{0}d\lambda\right)$.\end{sllem}
\begin{proof}[Proof of the lemma]
According to \cite{CDM}, Section 3, we have
\[
P_{\omega_{0}}^{\Omega}=R\circ P^{\widetilde{\Omega}}\circ I,
\]
where $P^{\widetilde{\Omega}}$ is the standard Bergman projection
of $\widetilde{\Omega}$.
\begin{sllem}
\label{lem:complex-transversality-omega-tilde}There exists a vector
field $W=\sum a_{i}\frac{\partial}{\partial w_{i}}+b_{i}\frac{\partial}{\partial\overline{w}_{i}}$
with coefficients $a_{i}$ and $b_{i}$ in $\mathcal{C}^{\infty}(\overline{\widetilde{\Omega}})$
such that $T+W$ is smooth in $\overline{\widetilde{\Omega}}$,
tangential and complex transversal to $\partial\widetilde{\Omega}$.\end{sllem}
\begin{proof}[Proof of \lemref{complex-transversality-omega-tilde}]
This a very simple calculus. $T+W=\varphi T_{\rho_{0}}+L+W=\varphi T_{r}-\varphi T_{\left|w\right|^{2q}}+L+W$,
where $r$ denotes the defining function of $\widetilde{\Omega}$.
As $\varphi T_{r}$ is tangential and complex transversal to $\partial\widetilde{\Omega}$
it is enough to see that the coefficients of $W$ can be chosen so
that the $\left(1,0\right)$ and $\left(0,1\right)$ parts of $-\varphi T_{\left|w\right|^{2q}}+L+W$
are both tangential to $\partial\widetilde{\Omega}$. For example,
the $\left(1,0\right)$ part of this vector field is
\[
-q\varphi\left|w\right|^{2q-2}\sum w_{i}\frac{\partial}{\partial w_{i}}+L_{1}+\sum a_{i}\frac{\partial}{\partial w_{i}},
\]
and it is tangent to $\partial\widetilde{\Omega}$ if
$q^{2}\varphi\left|w\right|^{4q-2}-L_{1}\rho_{0}\equiv\sum qa_{i}\left|w\right|^{2q-2}\overline{w}_{i}$
on $\partial\widetilde{\Omega}$. As $L_{1}$ is tangential to $\partial\Omega$,
$L_{1}\rho_{0}$ vanishes at $\partial\Omega$ and there exists a
function $\psi\in\mathcal{C}^{\infty}(\overline{\Omega})$ such that
$-L_{1}\rho_{0}=\psi_{1}\left(-\rho_{0}\right)$. If $\left(z,w\right)\in\partial\widetilde{\Omega}$
then $-\rho_{0}(z)=\left|w\right|^{2q}$, and, it suffices to choose
\[
a_{i}=\frac{1}{q}w_{i}\left[q^{2}\varphi\left|w\right|^{2q-2}+\psi_{1}\right].
\]
Similarly, the $\left(0,1\right)$ part of $-\varphi T_{\left|w\right|^{2q}}+L+W$
is tangent to $\partial\widetilde{\Omega}$ choosing
$b_{i}=\frac{1}{q}\overline{w}_{i}\left[q^{2}\varphi\left|w\right|^{2q-2}+\psi_{2}\right]$.
\end{proof}
Let us now finish the proof of \lemref{P-omega-s_W-T-s}. By Theorem
1.1 of \cite{HMS14}, for any non negative integer $k$, $P^{\widetilde{\Omega}}$
maps continuously $L_{k,T+W}^{2}(\widetilde{\Omega})$ into $L_{k}^{2}(\widetilde{\Omega})$.
As $\left(T+W\right)^{l}(I(u))=T^{l}(u)$, for each $u\in L_{k,T}^{2}\left(\Omega,\omega_{0}d\lambda\right)$,
we have $I(u)\in L_{k,T+W}^{2}(\widetilde{\Omega})$, and, for $s=k$,
the Lemma follows (\ref{Reg_R_W-k}) of \lemref{I_and_R_global}.
The general case $s\geq0$ is therefore obtained by interpolation.
\end{proof}
\medskip{}

Now we use the formula of \propref{Relation-P-omega_P-omega-0} to
prove \thmref{Reformulation-General-Sobolev}, by induction, for $s\in\left\{ 0,\ldots,k\right\} $,
the general case $s\in\left[0,k\right]$ being then a consequence
of the interpolation theorem. Let us assume the Theorem true for $s-1$
, $0<s\leq k$ and let us prove it for $s$. Let $N$ be an integer
whose inverse is smaller than the index of subellipticity of the $\overline{\partial}$-Neumann
problem of $\widetilde{\Omega}$ (recall that we show in \cite{CDM}
that $\widetilde{\Omega}$ is of finite type).
Let $u\in L_{s,T}^{2}\left(\Omega,\omega_{0}d\lambda\right)$.
To prove that $P_{\omega}^{\Omega}(u)\in L_{s}^{2}\left(\Omega,\omega_{0}d\lambda\right)$,
let us prove, by induction over $l\in\left\{ 0,1,\ldots,N\right\} $
that $P_{\omega}^{\Omega}(u)\in L_{s-1+\nicefrac{l}{N}}^{2}\left(\Omega,\omega_{0}d\lambda\right)$.
Assume $P_{\omega}^{\Omega}(u)\in L_{s-1+\nicefrac{l}{N}}^{2}\left(\Omega,\omega_{0}d\lambda\right)$,
$l\leq N-1$. As $\eta\in\mathcal{C}^{k+1}(\overline{\Omega})$, by
(\ref{Reg_I_W-s}) of \lemref{Lemma_I_R}, $I\left(P_{\omega}^{\Omega}(u)\overline{\partial}\left(\eta\right)\right)\in L_{s-1+\nicefrac{l}{N}}^{2}(\widetilde{\Omega})$.
By subelliptic estimates for the $\overline{\partial}$-Neumann problem
on $\widetilde{\Omega}$, 
\[
\left(\overline{\partial}^{*}\mathcal{N}_{\widetilde{\Omega}}\right)\circ I\left(P_{\omega}^{\Omega}(u)\overline{\partial}\left(\eta\right)\right)\in L_{s-1+\nicefrac{(l+1)}{N}}^{2}(\widetilde{\Omega}).
\]
By (\ref{Reg_R_W-s-T}) of \lemref{Lemma_I_R}, 
\[
R\circ\left(\overline{\partial}^{*}\mathcal{N}_{\widetilde{\Omega}}\right)\circ I\left(P_{\omega}^{\Omega}(u)\overline{\partial}\left(\eta\right)\right)\in L_{s-1+\nicefrac{(l+1)}{N},T_{\rho_{0}}}^{2}\left(\Omega,\omega_{0}d\lambda\right).
\]
By \lemref{P-omega-s_W-T-s} $P_{\omega_{0}}^{\Omega}\left(\eta u\right)\in L_{s}^{2}\left(\Omega,\omega_{0}d\lambda\right)$,
thus, as $\eta^{-1}\in\mathcal{C}^{k+1}(\overline{\Omega})$, \propref{Relation-P-omega_P-omega-0}
gives 
\[
P_{\omega}^{\Omega}(u)\in L_{s-1+\nicefrac{(l+1)}{N},T_{\rho_{0}}}^{2}\left(\Omega,\omega_{0}d\lambda\right),
\]
and, as $P_{\omega_{0}}^{\Omega}\circ P_{\omega}^{\Omega}=P_{\omega}^{\Omega}$,
\lemref{P-omega-s_W-T-s} implies $P_{\omega}^{\Omega}(u)\in L_{s-1+\nicefrac{(l+1)}{N}}^{2}\left(\Omega,\omega_{0}d\lambda\right)$
finishing the proof.\end{proof}
\begin{proof}[Proof of the corollary of \thmref{General_Sobolev_estimates}]
It is enough to see that, if $l_{r}$ is a positive integer such that
$2l_{r}\geq r$ then, for any integer $k\geq l_{r}+1$, for $u\in L_{k,T}^{2}\left(\Omega,\omega_{0}d\lambda\right)$
we have $P_{\omega}^{\Omega}(u)\in L_{k-l_{r}}^{2}(\Omega)$. But
this is a consequence of the theorem and of Theorem 1.1 of \cite{CK03}:
$P_{\omega}^{\Omega}(u)\in L_{k}^{2}\left(\Omega,\omega_{0}d\lambda\right)\subset L_{k}^{2}\left(\Omega,\delta_{\partial\Omega}^{2l_{r}}d\lambda\right)$,
$\delta_{\partial\Omega}$ being the distance to the boundary of $\Omega$,
and a harmonic function in $L_{k}^{2}\left(\Omega,\delta_{\partial\Omega}^{2l_{r}}d\lambda\right)$
is in $L_{k-l_{r}}^{2}(\Omega)$.
\end{proof}
\begin{rem*}
\quad\mynobreakpar
\begin{enumerate}
\item As noted in \cite{Herbig-McNeal_Smoo-Berg_2012}, $\cap_{k\in\mathbb{N}}L_{k,T_{\varphi\rho_{0}}}^{2}\left(\Omega,\omega_{0}d\lambda\right)$
is, in general, strictly larger than $\mathcal{C}^{\infty}(\overline{\Omega})$.
\item In Remark 4.1 (2) of \cite{CDM} we notice that, if $\Omega$ is a
smoothly bounded pseudoconvex domain in $\mathbb{C}^{n}$ (not assumed
of finite type) admitting a defining function $\rho_{1}$ pluri-subharmonic
in $\Omega$ then, by a result of H. Boas and E. Straube (\cite{Boas-Straube-def-psh91})
the weighted Bergman projection $P_{\omega_{1}}^{\Omega}$, $\omega_{1}=\left(-\rho_{1}\right)^{r}$,
maps continuously the Sobolev space $L_{s}^{2}\left(\Omega,\omega_{1}d\lambda\right)$
into themselves.\\
Then, using Theorem 1.1 of \cite{HMS14}, the proof of \lemref{P-omega-s_W-T-s}
shows that $P_{\omega_{1}}^{\Omega}$ maps continuously $L_{s,T_{\rho_{1}}}^{2}\left(\Omega,\omega_{1}d\lambda\right)$
into $L_{s}^{2}\left(\Omega,\omega_{1}d\lambda\right)$, $s\geq0$.
Moreover, the arguments of the proof of the above corollary show that
$P_{\omega_{1}}^{\Omega}$ maps continuously $\cap_{k\in\mathbb{N}}L_{k,T_{\rho_{1}}}^{2}\left(\Omega,\omega_{1}d\lambda\right)$
into $\mathcal{C}^{\infty}(\overline{\Omega})$.\\
If $\rho$ is another defining function of such a domain, we do not
know if $P_{\nu}^{\Omega}$, $\nu=\left(-\rho\right)^{r}$, is $L_{s}^{2}\left(\Omega,\nu d\lambda\right)$
regular.
\end{enumerate}
\end{rem*}

\begin{proof}[Proof of \thmref{Cinfini-estimate-conj-holo}]
It is very similar to the proof of \thmref{Reformulation-General-Sobolev}.
First, by Theorem 1.10 of \cite{HMS14}, $P^{\widetilde{\Omega}}(I(\eta fh))\in\mathcal{C}^{\infty}(\overline{\widetilde{\Omega}})$
so that $P_{\omega_{0}}^{\Omega}(\eta fh)\in\mathcal{C}^{\infty}(\overline{\Omega})$.
Then, by induction, the arguments used in the proof of \thmref{Reformulation-General-Sobolev}
show that, for all non negative integer $k$, $P_{\omega}^{\Omega}(fh)\in L_{k}^{2}\left(\Omega,\omega d\lambda\right)$.
Then, arguing as in the proof of the corollary of \thmref{Reformulation-General-Sobolev}
we conclude that $P_{\omega}^{\Omega}(fh)\in L_{k-l_{r}}^{2}(\Omega)$
which completes the proof.
\end{proof}

\section{\label{sec:Proof-of-theorem-C2-convex}Proof of \texorpdfstring{\thmref{Estimates-C2-rank-Levi}}{Theorem 2.3}
and \texorpdfstring{\propref{Estimates_for_Convex_domains}}{Proposition 2.1}}

The proof, based on the formula of \propref{Relation-P-omega_P-omega-0}
and on estimates for solutions of the $\overline{\partial}$-equation,
is very similar to the one given in the previous section. As only
the case of domains with rank of the Levi form $\geq n-2$ is general
(due to the restriction on the defining function for convex domains)
we will give the proof with some details in this case and only indicate
the steps for the convex case.

As in the preceding section we obtain the Sobolev estimates of \thmref{Estimates-C2-rank-Levi}
proving a stronger directional estimate:
\begin{stthm}
\label{thm:reformulation-C-2}Let $\Omega$ be as \thmref{Estimates-C2-rank-Levi}.
Let $\rho_{0}$, $\omega_{0}$ and $\omega$ be as at the end of \secref{Notations-and-main-results}.
Then:
\begin{enumerate}
\item Let $k$ be a non negative integer. Assume $\eta\in\mathcal{C}^{k+1}(\overline{\Omega})$.
Then, for $1<p<+\infty$ and $s\in\left[0,k\right]$ the weighted
Bergman projection $P_{\omega}^{\Omega}$ maps continuously the directional
weighted Sobolev space $L_{s,T_{\rho_{0}}}^{p}\left(\Omega,\omega_{0}d\lambda\right)$
into $L_{s}^{p}\left(\Omega,\omega_{0}d\lambda\right)$.
\item Let $\alpha\leq1$. Assume $\eta\in\mathcal{C}^{\left[\alpha\right]+1}(\overline{\Omega})$.
Then the weighted Bergman projection $P_{\omega}^{\Omega}$ maps continuously
the lipschitz space $\Lambda_{\alpha}(\Omega)$ into itself.
\end{enumerate}
\end{stthm}

We know (\cite{CDM}) that the Levi form of the domain $\widetilde{\Omega}$
is locally diagonalizable at every point of $\partial\widetilde{\Omega}$.
Thus we use the estimates for the $\overline{\partial}$-Neumann problem
obtained by C. L. Fefferman, J. J. Kohn and M. Machedon in 1990 and
by K. Koenig in 2004 for these domains:
\begin{stthm}[\cite{F-K-M-d-bar-Diag-Levi-Form,Koe04}]
\label{thm:d-bar-Levi-diag}Let $D$ be a smoothly bounded pseudoconvex
domain in $\mathbb{C}^{n}$ of finite type whose Levi form is locally
diagonalizable at every boundary point. Then there exists a positive
integer $N$ such that:
\begin{enumerate}
\item For every $\alpha\geq0$, $\overline{\partial}^{*}\mathcal{N}_{D}$
maps continuously the lipschitz space $\Lambda_{\alpha}(D)$ into
$\Lambda_{\alpha+\nicefrac{1}{N}}(D)$;
\item For $1<p<+\infty$ and $s\geq0$, $\overline{\partial}^{*}\mathcal{N}_{D}$
maps continuously the Sobolev space $L_{s}^{p}(D)$ into $L_{s+\nicefrac{1}{N}}^{p}(D)$;
\item For $1<p<+\infty$, $\overline{\partial}^{*}\mathcal{N}_{D}$ maps
continuously $L^{p}(D)$ into $L^{p+\nicefrac{1}{N}}(D)$;
\item For $p$ sufficiently large $\overline{\partial}^{*}\mathcal{N}_{D}$
maps continuously $L^{p}(D)$ into $\Lambda_{0}(D)$.
\end{enumerate}
\end{stthm}

The first statement is explicitly stated in \cite{F-K-M-d-bar-Diag-Levi-Form},
for $N$ strictly larger than the type of $D$, for the $\overline{\partial}_{b}$-Neumann
problem at the boundary, and exactly stated in \cite{Koe04} (Corollary
6.3, p. 286). In \cite{Koe04} it is also proved that $\overline{\partial}^{*}\mathcal{N}_{D}$
maps continuously the Sobolev space $L_{s}^{p}(D)$ into $L_{s+\nicefrac{1}{m}-\varepsilon}^{p}(D)$,
where $m$ is the type of $D$ and $\varepsilon>0$. Therefore the
third and fourth statements of the theorem follow the Sobolev embedding
theorems (see, for example, \cite{AF03}).

\smallskip{}

We need also directional Sobolev estimates for the standard Bergman
projection $P^{\widetilde{\Omega}}$. Such estimates have been obtained
for finite type domains in $\mathbb{C}^{2}$ by A. Bonami, D.-C. Chang
and S. Grellier (\cite{BCG96}) and by D.-C. Chang and B. Q. Li (\cite{CDC97})
in the case of decoupled domains of finite type in $\mathbb{C}^{n}$.

Following the proof of Lemma 3.4 of \cite{Charpentier-Dupain-Szego-Barcelone}
but using the integral curve of the real normal to the boundary of
$D$ as in the proof of Theorem 4.2.1 of \cite{BCG96}, instead of
a coordinate in a special coordinate system (also used in \cite{McNeal-Stein-Bergman}),
we easily write $\nabla^{k}P^{D}=\sum P_{i}^{D}T_{d}^{i}$ with ``good''
operators $P_{i}^{D}$ and obtain the following estimate for $P^{D}$:
\begin{stthm}
\label{thm:directional-L-p-Sobolev-loc-diag-cvx}Let $D$ be a smoothly
bounded pseudoconvex domain of finite type in $\mathbb{C}^{n}$ whose
Levi form is locally diagonalizable at every point of $\partial D$.

If $d$ is a defining function of $D$, let $T_{d}=\sum_{i}\frac{\partial d}{\partial\overline{z_{i}}}\frac{\partial}{\partial z_{i}}-\frac{\partial d}{\partial z_{i}}\frac{\partial}{\partial\overline{z_{i}}}$.
Then, for $1<p<+\infty$ and $s\geq0$, the Bergman projection $P^{D}$
of $D$ maps continuously the space $L_{s,T_{d}}^{p}(D)$ into $L_{s}^{p}(D)$.
\end{stthm}
\smallskip{}

\begin{proof}[Proof of \thmref{reformulation-C-2}]
Let us first prove the weighted $L^{p}$ regularity of $P_{\omega}^{\Omega}$.
Let $u\in L^{p}\left(\Omega,\omega d\lambda\right)$. Assume for the
moment $p>2$. Let $N_{p}$ be an integer such that $\nicefrac{p-2}{N_{p}}<\nicefrac{1}{N}$
where $N$ is the integer of \thmref{d-bar-Levi-diag} and let us
prove, by induction over $l\in\left\{ 0,\ldots,N_{p}\right\} $ that
$P_{\omega}^{\Omega}(u)\in L^{2+\nicefrac{l(p-2)}{N_{p}}}\left(\Omega,\omega d\lambda\right)$.
Assume that $P_{\omega}^{\Omega}(u)\in L^{2+\nicefrac{l(p-2)}{N_{p}}}\left(\Omega,\omega d\lambda\right)$
for $l<N_{p}$. Then \lemref{I_and_R_global} and \thmref{d-bar-Levi-diag}
give $\overline{\partial}^{*}\mathcal{N}_{\widetilde{\Omega}}\circ I\left(P_{\omega}^{\Omega}(u)\overline{\partial}\eta\right)\in L^{p+\nicefrac{(l+1)(p-2}{N_{p}}}(\widetilde{\Omega})$,
and the second part of \lemref{I_and_R_global} gives the result.
The $L^{p}$ regularity of $P_{\omega}^{\Omega}$ for $1<p<2$ is
then obtained using the fact that $P_{\omega}^{\Omega}$ is self-adjoint.

The $\Lambda_{\alpha}$ regularity is proved similarly. Suppose $u\in\Lambda_{\alpha}(\Omega)$.
Then $u$ belongs to all $L^{p}\left(\Omega,\omega d\lambda\right)$
spaces, $p<+\infty$, and, the $L^{p}\left(\Omega,\omega d\lambda\right)$
regularity of $P_{\omega}^{\Omega}$, \lemref{I_and_R_global} and
the last assertion of \thmref{d-bar-Levi-diag} show that $\overline{\partial}^{*}\mathcal{N}_{\widetilde{\Omega}}\circ I\left(P_{\omega}^{\Omega}(u)\overline{\partial}\eta\right)\in\Lambda_{0}\left(\widetilde{\Omega}\right)$
therefore $P_{\omega}^{\Omega}(u)\in\Lambda_{0}(\Omega)$. Then, using
the first assertion of \thmref{d-bar-Levi-diag} it is easy to prove,
by induction, that $P_{\omega}^{\Omega}(u)\in\Lambda_{\nicefrac{l\alpha}{N_{\alpha}}}(\Omega)$,
$l\in\left\{ 1,\ldots,N_{\alpha}\right\} $, where $N_{\alpha}$ is
a sufficiently large integer.

For the $L_{s}^{p}\left(\Omega,\omega d\lambda\right)$ regularity,
we deduce from \thmref{directional-L-p-Sobolev-loc-diag-cvx} the
following extension of \lemref{P-omega-s_W-T-s}:
\begin{sllem}
$P_{\omega_{0}}^{\Omega}$ maps continuously the space $L_{s,T_{\rho_{0}}}^{p}\left(\Omega,\omega_{0}d\lambda\right)$
into $L_{s}^{p}\left(\Omega,\omega_{0}d\lambda\right)$.
\end{sllem}
As we already know that $P_{\omega}^{\Omega}$ maps $L^{p}\left(\Omega,\omega_{0}d\lambda\right)$
into itself, the proof of the $L_{s,T_{\rho_{0}}}^{p}\left(\Omega,\omega_{0}d\lambda\right)$-$L_{s}^{p}\left(\Omega,\omega_{0}d\lambda\right)$
regularity of $P_{\omega}^{\Omega}$ is identical to the end of the
proof of \thmref{Reformulation-General-Sobolev}.
\end{proof}
\medskip{}

For the convex domains considered in \propref{Estimates_for_Convex_domains}
the scheme of the proof is strictly identical.

For the $L^{p}$ estimate we use the estimates for solutions of the
$\overline{\partial}$-equation given by A. Cumenge:
\begin{thm*}[\cite{Cumenge-estimates-holder,Cumenge-Navanlinna-convex}]
Let $D$ be a smoothly bounded convex domain in $\mathbb{C}^{l}$
of finite type $\tau_{D}$. Then:
\begin{enumerate}
\item For $1\leq p<\tau_{D}l+2$ the restriction of $\overline{\partial}^{*}\mathcal{N}_{D}$
to $\overline{\partial}$-closed $\left(0,1\right)$-forms maps continuously
$L_{\left(0,1\right)}^{p}(D)\cap\ker\overline{\partial}$ into $L^{s}(D)$
with $\nicefrac{1}{s}=\nicefrac{1}{p}-\nicefrac{1}{\tau_{D}l+2}$;
\item For $\tau_{D}l+2<p\leq+\infty$, the restriction of $\overline{\partial}^{*}\mathcal{N}_{D}$
to $\overline{\partial}$-closed $\left(0,1\right)$-forms maps continuously
$L_{\left(0,1\right)}^{p}(D)\cap\ker\overline{\partial}$ into the
lipschitz space $\Lambda_{\alpha}(D)$ with $\alpha=\nicefrac{1}{\tau_{D}}-\nicefrac{(l+\nicefrac{2}{\tau_{D}})}{p}$.
\end{enumerate}
\end{thm*}

For the lipschitz estimate, as the type of $\widetilde{\Omega}$ is
larger than the type of $\Omega$, we need a general lipschitz spaces
estimate for the solutions of the $\overline{\partial}$-equation.
Using techniques developed in \cite{McNeal-Stein-Bergman,McNeal-Stein-Szego,Charpentier-Dupain-Szego-Barcelone}
and formulas introduced in \cite{CDMb} the following result can be
proved for lineally convex domains of finite type (as the detailed
proof is long, technical and not new, we will note write it here):
\begin{thm*}
Under the conditions of the preceding theorem, for $\alpha\geq0$,
the restriction of $\overline{\partial}^{*}\mathcal{N}_{D}$ to $\overline{\partial}$-closed
$\left(0,1\right)$-forms maps continuously the lipschitz spaces $\Lambda_{\alpha}(D)$
into $\Lambda_{\alpha+\nicefrac{1}{\tau_{D}}}$.
\end{thm*}
Finally for the Sobolev $L_{s}^{p}\left(\Omega,\omega d\lambda\right)$
estimate, using techniques similar to those used in the previous estimate
it can be shown that, for $1<p<+\infty$ and $s\geq0$, the restriction
of $\overline{\partial}^{*}\mathcal{N}_{D}$ to $\overline{\partial}$-closed
$\left(0,1\right)$-forms maps continuously $L_{s}^{p}(D)$ into $L_{s+\nicefrac{1}{\tau_{D}}}^{p}(D)$.

\bibliographystyle{amsalpha}

\providecommand{\bysame}{\leavevmode\hbox to3em{\hrulefill}\thinspace}
\providecommand{\MR}{\relax\ifhmode\unskip\space\fi MR }
\providecommand{\MRhref}[2]{%
  \href{http://www.ams.org/mathscinet-getitem?mr=#1}{#2}
}
\providecommand{\href}[2]{#2}

\end{document}